\documentclass[12pt, reqno, a4paper, onesided]{amsart}
\usepackage{graphicx, epsfig, psfrag}
\usepackage{amsfonts,amsmath,amssymb,amsbsy,amsthm}
\usepackage{bm}
\usepackage{color}
\usepackage{float}
\usepackage{hyperref}
\usepackage{mathrsfs}
\usepackage{subfigure}
\usepackage{bbm}
\usepackage{longtable}

\usepackage{pdfsync}

\usepackage[top=3cm, left=3cm, right=3cm, bottom=3cm]{geometry}

\usepackage{environments}

\renewcommand{\paragraph}[1]{\subsubsection{#1}}

\renewcommand{\cases}[1]{\left\{ \begin{array}{rl} #1 \end{array} \right.}

\def\XXint#1#2#3{{\setbox0=\hbox{$#1{#2#3}{\int}$ }
\vcenter{\hbox{$#2#3$ }}\kern-.6\wd0}}

\def\b{\big}
\def\B{\Big}
\def\bg{\bigg}

\def\sep{\,|\,}
\def\bsep{\,\b|\,}

\def\id{{\rm id}}

\def\R{\mathbb{R}}
\def\N{\mathbb{N}}
\def\Z{\mathbb{Z}}

\def\WW{W}

\def\LL{L}

\def\dx{\,{\rm d}x}
\def\dy{\,{\rm d}y}

\def\<{\langle}
\def\>{\rangle}

\def\mA{{\sf A}}

\def\D{\nabla}

\def\loc{{\rm loc}}

\def\L{{\Z^d}}

\def\Ws{\mathscr{W}}
\def\Us{\mathscr{U}}

\def\Usz{\Us_0}

\def\Ys{\mathscr{Y}}

\def\yA{y_\mA}

\def\Om{{\R^d}}

\def\tilu{\tilde u}
\def\tilv{\tilde v}

\def\tilw{\tilde w}

\def\barv{\bar v}
\def\barw{\bar w}
\def\baru{\bar u}

\def\zz{\bar{\zeta}}

\def\zzz{\tilde{\zeta}}

\def\Stil{\tilde{\mathscr{S}}}

\def\QQ{\mathscr{Q}}

\def\TT{\mathscr{T}}

\def\ii{\imath}

\definecolor{cocol}{rgb}{0.7, 0, 0}

\definecolor{ascol}{rgb}{0, 0, 0.7}

\def\Conv{\mathscr{C}}
\def\Fourier{\mathscr{F}}

\begin{document}

\title[Interpolants of Lattice Functions]{Interpolants of Lattice
  Functions for the Analysis of Atomistic/Continuum Multiscale
  Methods}

\author{C. Ortner}
\address{C. Ortner\\ Mathematics Institute \\ Zeeman Building \\
  University of Warwick \\ Coventry CV4 7AL \\ UK}
\email{c.ortner@warwick.ac.uk}

\author{A. V. Shapeev}
\address{A. V. Shapeev \\ School of Mathematics \\ University of
  Minnesota \\ 127 Vincent Hall \\ 206 Church St SE \\ Minneapolis \\
  MN 55455 \\ USA }
\email{alexander@shapeev.com}


\date{\today}

\thanks{This work was supported by the EPSRC Critical Mass Programme
  ``New Frontiers in the Mathematics of Solids'' (OxMoS) and by the
  EPSRC grant EP/H003096 ``Analysis of atomistic-to-continuum coupling
  methods''.}



\begin{abstract}
  We introduce a general class of (quasi-)interpolants of functions
  defined on a Bravais lattice, and establish several technical
  results for these interpolants that are crucial ingredients in the
  analysis of atomistic models and atomistic/continuum multi-scale
  methods.
\end{abstract}

\maketitle


\section{Introduction}
In this report, we describe several classes of continuous interpolants
or quasi-interpolants of lattice functions that have proven useful in
recent and ongoing work \cite{OrSh:inf, OrTh:2012, OrVK:bqce_1}
for the analysis of atomistic models of crystalline solids and for for
the construction and analysis of atomistic/continuum multiscale
methods \cite{Ortiz:1995a, Shenoy:1999a, OrtnerShapeev:2011pre,
  Or:2011a, Shapeev:2010a}.

A key ingredient that made the analyses in \cite{OrSh:inf,
  OrTh:2012, OrVK:bqce_1, OrtnerShapeev:2011pre, Or:2011a} possible is
the construction of continuous representations of discrete object. A
novel recent idea that allowed substantial progress in the direction
of such constructions is the {\em bond density lemma}
\cite{Shapeev:2010a}, which allows explicitly compute the ``density of
bonds'' in a certain subsets of $\R^d$. The lattice interpolants we
define and analyze in this paper are designed to go hand in hand with
these constructions.

Although many results related to the ones we present here can be found
in the numerical analysis literature, we have not found the precise
statements we require or the generality that we seek (as a special
case our analysis covers multi-linear and multi-cubic splines for
which many of our results are standard). Hence, we present a complete
derivation of all results. Moreover, our function space setting seems
new as well.



\section{Interpolation of Lattice Functions}
\label{sec:prelims}
We fix a space dimension $d \in \N$ and a range dimension $m \in
\N$. The goal of this section, and core of this paper, is to introduce
spaces of discrete and continous functions, respectively, on the {\em
  continous domain} $\R^d$ and the {\em discrete domain} $\Z^d$, and
to provide tools to transition between these two classes.


\subsection{Interpolants of lattice functions}
\label{sec:interp}
We denote the set of all vector-valued lattice functions by
\begin{displaymath}
  \Us := \b\{ v : \Z^d \to \R^m \b\}.
\end{displaymath}
To facilitate the transition between continuous and discrete maps we
introduce two (quasi-)interpolants of lattice functions, and will
later introduce a third nodal interpolant.

Our starting point is a nodal basis function $\zz \in
\WW^{1,\infty}(\R^d; \R)$ associated with the origin; that is $\zz(0)
= 1$ and $\zz(\xi) = 0$ for all $\xi \in \Z^d \setminus \{0\}$. There
is considerable freedom in the choice of $\zz$, however, certain
choices are particularly natural; see Section \ref{sec:examples}. The
nodal basis function associated with $\xi \neq 0$ is the shifted
function $\zz(\bullet - \xi)$, which yields the interpolant
\begin{equation}
  \label{eq:interp:S1_interp}
  \barv(x) := \sum_{\xi \in \Z^d} v(\xi) \zz(x - \xi), \qquad
  \text{for } v \in \Us.
\end{equation}

Next, we define a quasi-interpolant obtained through convolution of
$\barv$ with $\zz$:
\begin{equation}
  \label{eq:interp:defn_tildev}
  \tilv(x) := (\zz \ast \barv)(x) = \int_{\R^d} 
  \zz(x - x') \barv(x') \dx'.
\end{equation}
We call $\tilv$ a quasi-interpolant since, in general, $v(\xi) \neq
\tilv(\xi)$. This can be seen from the alternative definition
\begin{equation}
  \label{eq:interp:defn_tildev_2}
  \tilv(x) = \sum_{\xi \in \Z^d} v(\xi) \zzz(x-\xi), 
  \qquad \text{where} \quad  \zzz(x) := (\zz \ast \zz)(x).
\end{equation}

For future reference we define the support sets
\begin{displaymath}
  \bar{\omega}_\xi := {\rm supp}\b(\zz(\bullet - \xi)\b), \quad
  \text{and} \quad
  \tilde{\omega}_\xi := {\rm supp}\b(\zzz(\bullet - \xi)\b).
\end{displaymath}

Throughout, we make the following standing assumptions on $\zz$:
\begin{itemize}
\item[(Z1)] {\it Regularity: } $\zz \in W^{1,\infty}$
\item[(Z2)] {\it Locality: } $\bar{\omega}_\xi$ (and hence
  $\tilde{\omega}_\xi$) is compact
\item[(Z3)] {\it Order: } $\sum_{\xi \in \Z^d} \zz(x - \xi) (a + b
  \cdot \xi) = a + b \cdot x$ for all $a \in \R, b \in \R^d$.
\item[(Z4)] {\it Invertibility: } $\zz(0) = 1$ and $\zz(\xi) = 0$ for
  all $\xi \in \Z^d \setminus \{0\}$.
\end{itemize}

\begin{remark}
  1. The assumption (Z3) ensures that interpolants of affine functions
  are again affine: if $u(\xi) = a + b \cdot \xi$, then $\baru(x) = a
  + b \cdot x$. When applied to continuous functions in
  \S~\ref{sec:interp:errest}, this ensures that this interpolant is
  second-order accurate. There, we will also see that a nodal
  interpolant from the space $\{\tilv\}$ is third-order accurate.

  2. The assumption (Z4) is necessary not only to ensure that $\barv$
  interpolates $v$, but also to ensure that the operation $v \mapsto
  \barv$ is invertible. Indeed, the extended hat function in 1D
  \[
  \zz(x) = \cases{
    1/3 & -1\leq x \leq 1 \\
    1/3 (2+x) & -2 \leq x \leq 1 \\
    1/3 (2-x) & 1 \leq x \leq 2 \\
    0 & |x|>2
  }
  \]
  satisfies (Z1)--(Z3), yet for the function $u = (-1,0,1)_{\rm per}$
  (periodic repetition of $(-1,0,1)$) we have $\baru \equiv 0$.
\end{remark}

\begin{lemma}
  \label{th:interp:regul}
  Let $v \in \Us$, then $\barv \in \WW^{1,\infty}_\loc(\R^d; \R^m)$
  and $\tilv \in \WW^{3,\infty}_\loc(\R^d; \R^m)$.
\end{lemma}
\begin{proof}
  The basis function $\zz$ has compact support and belongs to
  $\WW^{1,\infty}$. Hence, in every compact subset of $\R^d$, $\barv$
  is a finite linear combination of Lipschitz functions, which implies
  that $\barv \in \WW^{1,\infty}_\loc$.
 
  A straightforward computation shows that
  \begin{align*}
    \D^2 \zzz(x) = \int \D\zz(x - x') \otimes \D\zz(x') \dx,
  \end{align*}
  hence $\D^2\zzz$ is Lipschitz continuous and in particular $\zzz \in
  \WW^{3,\infty}$ with compact support. The same argument as above
  applies to prove that $\tilv \in \WW^{3,\infty}_\loc$.
\end{proof}

\subsection{Examples} 
\label{sec:examples}
\begin{enumerate}
\item {\it Q1-interpolation: } possibly the most natural choice for
  the basis function $\zz$ is the tensor product
  \begin{equation} \label{eq:q1_interpolation}
    \zz(x) := \prod_{i = 1}^d \max\b(0, 1 - |x_i|\b).
  \end{equation}
  In this case, ${\rm supp}(\zz) = [-1,1]^d$ is compact and $\zz$ is
  piecewise multi-linear; $\{ \barv \sep v \in \Us \}$ is the space of
  tensor product linear B-splines on the grid~$\QQ$, and $\{\tilv \sep
  v \in \Us\}$ is the space of cubic tensor product B-splines on the
  grid~$\QQ$ (see \cite{Hollig2003}).

\item {\it P1-interpolant: } In some cases, such as in
  \cite{OrVK:bqce_1}, it is convenient if $\D \bar u$ is piecewise
  constant. This is also possible within our framework. Let $\TT$ be a
  regular simplicial partition of $\R^d$, which is translation
  invariant ($\xi+\TT = \TT$ for all $\xi \in \Z^d$) and symmetric
  about the origin ($-\TT = \TT$), and let $\zz$ denote the P1-nodal
  basis function with respect to this partition. Then it is easy to
  see that $\zz$ satisfies (Z1)--(Z4).

  For $d = 1, 2$ such sub-divisions are straightforward to
  construct. For $d = 3$, the subdivision of the cube shown in Figure
  \ref{fig:tet} yields a partition $\TT$ with the desired properties.
\end{enumerate}

\begin{figure}[t]
  \includegraphics[height=5cm]{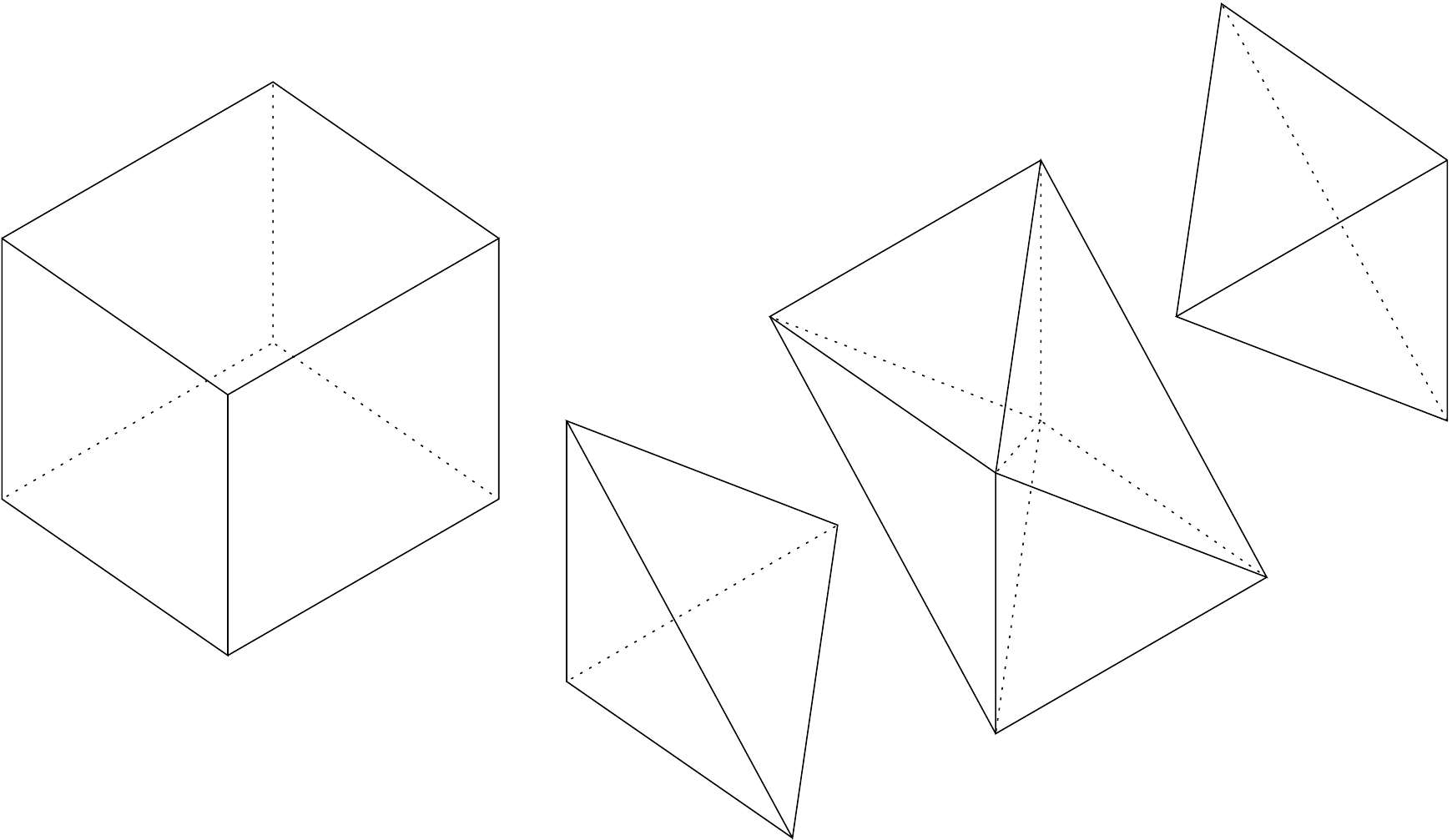}
  \caption{\label{fig:tet} Subdivision of the unit cube leading to a
    regular simplicial partition of $\R^d$ that is translation
    invariant and symmetric about the origin.}
\end{figure}

\subsection{Basic stability results}
\label{sec:interp_aux}
In the following theorem we collect all basic stability results for
the first two lattice interpolants. In these statements the norms may
take the value $+\infty$.

\begin{theorem}
  \label{th:norm_equiv_Lp}
  There exist constants $c_0, c_1, c_2 > 0$, independent of $p$, such
  that
  \begin{equation}
    \label{eq:interp:norm_equiv_Lp}
    c_0c_1 \| u \|_{\ell^p} \leq c_1 \|\tilu \|_{\ell^p} \leq \| \tilu \|_{\LL^p} 
    \leq \| \baru \|_{\LL^p} \leq c_2 \| u \|_{\ell^p}
    \qquad \forall u \in \Us, \quad p \in [1, \infty).
  \end{equation}
  Moreover, there exists a constant $c_1' > 0$, independent of $p$,
  such that
  \begin{equation}
    \label{eq:interp_stab}
    c_1' \| \D \baru \|_{\LL^p(\Om)} 
    \leq \| \D \tilu \|_{\LL^p(\Om)}
    \leq \| \D \baru \|_{\LL^p(\Om)}
    \qquad \forall u \in \Us, \quad p \in [1, \infty).
  \end{equation}
  All inequalities in \eqref{eq:interp:norm_equiv_Lp} and
  \eqref{eq:interp_stab}, except the first in each display, also hold
  for $p = \infty$.
\end{theorem}

\medskip
We establish this result in a series of lemmas.

\begin{lemma}
  \label{th:seminorm_lemma}
  Let $\mathscr{X}$ be a finite-dimensional vector space and let $S_1,
  S_2 : \mathscr{X} \to \R$ be semi-norms. If ${\rm ker}(S_2) \subset
  {\rm ker}(S_1)$ then there exists a constant $C > 0$ such that
  $S_1(u) \leq C S_2(u)$ for all $u \in \mathscr{X}$.
\end{lemma}

\begin{lemma}
  \label{th:simple_stab}
  There exist constants $c_1, c_2 > 0$ such that
  \begin{displaymath}
    c_1 \| \tilu \|_{\ell^p} \leq \| \tilu \|_{\LL^p} 
    \leq \| \baru \|_{\LL^p} \leq c_2 \| u \|_{\ell^p}    
  \end{displaymath}
\end{lemma}
\begin{proof}
  {\it 1. Proof of $\|\tilu\|_{\LL^p} \le \|\baru \|_{\LL^p}$: } This
  statement follows simply from the fact that $\tilu = \zz \ast \bar
  u$, where $\zz \geq 0$ and $\int\zz \dx = 1$ (see, e.g.,
  \cite[Sec. C.4]{EvansPDE}).

  {\it 2. Proof of $\| \baru \|_{\LL^p} \leq c_2 \| u \|_{\ell^p}$. }
  Fix $\xi \in \Z^d$ and let $\omega := \xi + [-1/2, 1/2]^d$. Let
  $S_1(u) := |u(\xi)|$ and $S_2(u) := \| \baru \|_{L^p(\omega)}$, then
  Lemma \ref{th:seminorm_lemma} implies that
  \begin{displaymath}
    \| \baru \|_{\LL^p} \leq c_2 \| u \|_{\ell^p}.
  \end{displaymath}

  An application of H\"{o}lder's inequality yields
  \begin{align*}
    \| \baru \|_{\LL^p}^p =~& \int_\Om \B| \sum_{\xi \in \L} u(\xi)
    \zz(x - \xi) \B|^p \dx \\
    \leq~& \int_{\R^d} \sum_{\xi \in \Z^d}
    |u(\xi)|^p |\zz(x-\xi)| \, \B(\sum_{\xi \in \Z^d} |\zz(x-\xi)|\B)^{p/p'} \dx.
  \end{align*}
  Since $\zz$ has compact support, we can bound
  \begin{displaymath}
    \B(\sum_{\xi \in \Z^d} |\zz(x-\xi)|\B)^{p/p'} \leq C,
  \end{displaymath}
  undidendently of $x$ and $p$. Hence, we obtain
  \begin{displaymath}
    \| \baru \|_{\LL^p}^p \leq C \int_{\R^d} \sum_{\xi \in \Z^d}
    |u(\xi)|^p |\zz(x-\xi)| \dx =C  \| u \|_{\ell^p} \| \zz \|_{L^1}.
  \end{displaymath}

  {\it 3. Proof of $c_1 \| \tilu \|_{\ell^p} \leq \|\tilu \|_{\LL^p}$: }
  If $\tilu = 0$ in $\xi + [-1/2,1/2]^d$ then $\tilu(\xi) = 0$, hence
  Lemma \ref{th:seminorm_lemma} implies that
  \begin{displaymath}
    |\tilu(\xi)| \leq C \| \tilu \|_{\LL^1(\xi+[-1/2,1/2]^d)} \leq C
    \| \tilu \|_{\LL^p(\xi+[-1/2,1/2]^d)},
  \end{displaymath}
  where $C$ is independent of $p$.  Summing over $\xi \in \L$ (and
  with suitable modification for $p = \infty$) yields
  \begin{displaymath}
    \| \tilu \|_{\ell^p} \leq C \|\tilu \|_{\LL^p}, \qquad \text{for }
    u \in \Us. \qedhere
  \end{displaymath}
\end{proof}

It remains to prove that $c_0 \| u \|_{\ell^p} \leq \| \tilu
\|_{\ell^p}$. We will first prove this for $p = 2$ and then use
Fourier analysis to extend the result to general $p$. It is convenient
to interpret this analysis in terms of the operator $\Conv :
\ell^p(\Z^d; \R^m) \to \ell^p(\Z^d; \R^m)$, $(\Conv u)(\xi) :=
\tilu(\xi)$.

\begin{lemma}
  \label{th:Conv_ell2}
  $\Conv : \ell^2 \to \ell^2$ is an isomorphism; that is, there exists
  $c_0 > 0$ such that
  \begin{displaymath}
    c_0 \| u \|_{\ell^2} \leq \| \tilu \|_{\ell^2} \qquad \forall u \in \Us.
  \end{displaymath}
\end{lemma}
\begin{proof}
  Let $u \in \ell^2$. Using the same argument as in step 3 of the
  proof of Lemma \ref{th:simple_stab} we obtain
  \begin{displaymath}
    \| u \|_{\ell^2} \leq c_0^{-1} \| \baru \|_{L^2};
  \end{displaymath}
  hence, we can estimate
  \begin{align*}
    \| u \|_{\ell^2}^2 \leq~& c_0^{-1} \| \baru \|_{L^2}^2 
    = c_0^{-1} \sum_{\xi\in\Z^d} u(\xi) \cdot \int \baru(x) \zz(x-\xi) \dx \\
    =\,& c_0^{-1} \sum_{\xi\in\Z^d} u(\xi) \cdot \tilu(\xi) 
    \leq c_0^{-1} \|u\|_{\ell^2} \|\tilu\|_{\ell^2}. \qedhere
  \end{align*}
\end{proof}

\begin{lemma}
  \label{th:Conv_ellp}
  For any $p \in [1, \infty)$, $\Conv : \ell^p \to \ell^p$ is an
  isomorphism; in particular, there exists $c_0 > 0$ such that
  \begin{displaymath}
     c_0 \| u \|_{\ell^p} \leq \| \tilu \|_{\ell^p} \qquad \forall u \in \Us.
  \end{displaymath}
\end{lemma}
\begin{proof}
  We use the representation
  \[
  \Conv u (\xi) = \sum_{\eta\in\Z^d} m(\xi-\eta) u(\eta),
  \]
  where $m(\xi-\eta) := \int \zz(\xi-x) \zz(\eta-x) \dx$.
  
  Next, we introduce a discrete Fourier transform $\Fourier : \ell^2
  \to L^2([0,1]^d)$,
  \[
  (\Fourier u)(\alpha) := \sum_{\xi \in \Z^d} u(\xi) e^{-2\pi\ii \alpha\cdot\xi}.
  \]
  With this definition $\Conv u = \Fourier^{-1} (\hat{m} \Fourier u)$,
  where $\hat{m}=\Fourier m \in L^2([0,1]^d)$, is the Fourier
  multiplier associated with $\Conv$.  Since $\Fourier$ is an
  isomorphism between $\ell^2$ and $L^2([0,1]^d)$, we can express
  \[
  \|\Conv^{-1}\|_{\ell^2} = (\min_\alpha \hat{m}(\alpha))^{-1},
  \]
  and from Lemma \ref{th:Conv_ell2} we can deduce that $\min_\alpha
  \hat{m}(\alpha) \geq c_0 > 0$.  Since $\hat{m}$ is analytic and
  bounded below it follows that $\hat{g}(\alpha) :=
  (\hat{m}(\alpha))^{-1}$ is an analytic function on $[0,1]^d$.

  Next, we show that $g \in \ell^1$, where $g := \Fourier^{-1}
  \hat{g}$. Let $k > d/2$, then
  \begin{align*}
    \|g\|_{\ell^1}
    \leq~&
    \b\|(1+|\xi|^2)^{-k}\b\|_{\ell^2} \, \b\|(1+|\xi|^2)^k g(\xi)\b\|_{\ell^2}
    \\ \leq~&
    \|(1+|\xi|^2)^{-k}\|_{\ell^2} \, \|(\id - \Delta)^k \hat{g}\|_{L^2([0,1]^d)}
    .
  \end{align*}
  The first term is bounded since $k > d/2$, while the second term is
  bounded due to analyticity of $\hat{g}$.

  \def\Gs{\mathscr{G}}
  Since $g \in \ell^1$ it is easy to see that the operator $\Gs : \ell^p \to \ell^p$,
  \begin{displaymath}
    (\Gs u)(\xi)  := \sum_{\eta \in \Z^d} g(\xi-\eta) u(\eta),
  \end{displaymath}
  is well-defined, for all $p \in [1, \infty]$: For $p \in
  \{1,\infty\}$ it is well-known that $\|\Gs\|_{\ell^p} =
  \|g\|_{\ell^1}$ (where $\|\Gs\|_{\ell^p}$ is the
  $\ell^p$-operator norm), and hence, by the Riesz--Thorin
  interpolation theorem, $\|\Gs\|_{\ell^p} \leq \|g\|_{\ell^1}$
  for all $p \in [1,\infty]$.

  It only remains to show that $\Gs = \Conv^{-1}$. This follows from
  the fact that, by construction, $\Gs = \Conv^{-1}$ on $\ell^2$ and
  density of functions with compact support in $\ell^p$ for $p <
  \infty$.
\end{proof}

Lemmas \ref{th:simple_stab}--\ref{th:Conv_ellp} establishes
\eqref{eq:interp:norm_equiv_Lp}. We will now prove the corresponding
result for the gradients.

\def\cDulo{c_1'}
\begin{lemma}
  \label{th:interp:norm_equiv}
  There exists a constant $\cDulo > 0$, independent of $p$, such that
  \begin{equation}
    \label{eq:interp_stab}
    \cDulo \| \D \baru \|_{\LL^p} 
    \leq \| \D \tilu \|_{\LL^p}
    \leq \| \D \baru \|_{\LL^p}
    \qquad \forall u \in \Us, \quad p \in [1, \infty).
  \end{equation}
  The second bound also holds for $p = \infty$.
\end{lemma}
\begin{proof}
  {\it 1. Upper bound. } Since $\tilu = \zz \ast \baru$ it follows
  that $\D\tilu = \zz \ast \D \baru$. Hence, the upper bound follows,
  e.g., from \cite[Sec. C.4]{EvansPDE}.

  {\it 2. Lower bound. } Let $v \in \Us$ be defined by $v(\xi) :=
  \tilu(\xi)$, then
  \begin{equation}
    \label{eq:interp:stab:5}
    \| \D\baru \|_{\LL^p} \leq \| \D\baru - \D\barv \|_{\LL^p} 
    + \|\D\barv\|_{\LL^p}.
  \end{equation}

  {\it 2.1 Estimating $\|\D\barv\|_{\LL^p}$. } This term can be
  estimated by a simple semi-norm equivalence argument.
  Fix a hypercube $Q \in \QQ$ and define $S_1(u) := \|\D
  \bar{\tilu}\|_{\LL^p(Q)}$ and $S_2(u) := \| \D
  \tilu\|_{\LL^p(Q)}$. $S_1, S_2$ are seminorms on $\Us$ and involve
  only a finite number of degrees of freedom. Moreover, if $S_2(u) =
  0$, then $\tilu$ is constant in $Q$, and consequently, $S_1(u) =
  0$. According to Lemma~\ref{th:seminorm_lemma}, there exists a
  constant $C > 0$, such that $S_1(u) \leq C S_2(u)$ for all $\tilu
  \in \Stil$. Due to translation invariance, the constant must be
  independent of the choice of $Q$. Summing over all cubes we obtain
  \begin{equation}
    \label{eq:interp:stab:7}
    \|\D\barv\|_{\LL^p(\Om)} \leq C \| \D\tilu \|_{\LL^p(\Om)}.
  \end{equation}

  {\it 2.2 Estimating $\| \D\baru - \D\barv \|_{\LL^p}$. } Suppose,
  first, that $1 < p < \infty$, then 
  \begin{align*}
    \| \D\barw \|_{\LL^p}^p =~& \int_\Om \B| \sum_{\xi \in \L} w(\xi)
    \D\zz(x - \xi) \B|^p \dx \\
    \leq~& \int_\Om \B( \sum_{\xi \in \L} |w(\xi)|^p |\D\zz(x - \xi)|  \B) \B(
    \sum_{\xi \in \L} |\D\zz(x - \xi)| \B)^{p/p'} \dx.
  \end{align*}
  Since $|\D\zz| \leq 1$, and since the number of lattice sites $\xi$
  such that $|\D\zz(x-\xi)| > 0$ is bounded independently of $x\in
  \R^d$, we obtain
  \begin{displaymath}
    \| \D\barw \|_{\LL^p}^p \leq C_1 \sum_{\xi \in \L} |w(\xi)|^p
    \int_\Om |\D\zz(x-\xi)| \dx.
  \end{displaymath}
  Using again $|\D\zz| \leq 1$ and the compact support of
  $\D\zz$ we obtain
  \begin{displaymath}
   \| \D\barw \|_{\LL^p(\Om)} \leq C_2 \| w \|_{\ell^p(\L)},
  \end{displaymath}
  for some constant $C_2$ that is independent of $w$.  Applying
  \eqref{eq:interp:norm_equiv_Lp} and reinserting the definition of
  $w$, we arrive at
  \begin{equation}
    \label{eq:interp:stab:10}
    \| \D\baru - \D\barv \|_{\LL^p(\Om)} \leq 
    C_3 \| \tilu - \tilv \|_{\ell^p(\L)}.
  \end{equation}

  \def\tiltilu{\tilde{\tilde{u}}}

  Since $\tilv$ is a quasi-interpolant of $\tilu$ we can bound the
  right-hand side of \eqref{eq:interp:stab:10} in terms of $\D
  \tilu$.  
  For a fixed lattice site $\xi \in \L$, we now obtain
  \begin{displaymath}
    \tilu(\xi) - \tilv(\xi) = \tilu(\xi) - \tiltilu(\xi) = \int_\Om
    \b[ \tilu(\xi) - \tilu(x) \b] \zz(x - \xi) \dx.
  \end{displaymath}
  Since $\int_\Om  \zz(x - \xi) \dx = 1$, this implies
  \begin{equation}
    \label{eq:interp:stab:15}
     \b|\tilu(\xi) - \tilv(\xi)\b| \leq
     \B(\int_\Om \b| \tilu(\xi) - \tilu(x) \b|^p \zz(x - \xi)
     \dx\B)^{1/p}
     =: S_1(\tilu).
  \end{equation}
  Moreover, let 
  \begin{displaymath}
    S_2(\tilu) := \B( \int_\Om |\D \tilu|^p \zz(x - \xi) \dx
    \B)^{1/p},
  \end{displaymath}
  then $S_1, S_2$ are both semi-norms on $\Stil$. Clearly, if
  $S_2(\tilu) = 0$ then $\tilu$ is constant in $\bar{Q}_\xi$, which
  implies that also $S_1(\tilu) = 0$. Applying Lemma
  \ref{th:seminorm_lemma}, we obtain a constant $C_4 > 0$ such that
  \begin{displaymath}
    S_1(\tilu) \leq C_4 S_2(\tilu) \qquad \forall \tilu \in \Stil.
  \end{displaymath}
  Due to translation invariance, the constant $C_4$ is independent of
  $\xi$. Combining this result with \eqref{eq:interp:stab:15},
  \eqref{eq:interp:stab:10} and the fact that $\zz(\bullet-\xi), \xi
  \in \L$ is a partition of unity, we obtain
  \begin{align}
   \| \D\baru - \D\barv \|_{\LL^p}^p 
    \notag
    \leq~& C_3^p \sum_{\xi \in \L} \b|\tilu(\xi) - \tilv(\xi)\b|^p  \\
    \label{eq:interp:stab:20}
    \leq~& C_3^pC_4^p \sum_{\xi \in \L} 
    \int_\Om | \D\tilu |^p \zz(x - \xi)
    \dx = (C_3C_4)^p \|\D\tilu \|_{\LL^p(\Om)}^p.
  \end{align}
  This concludes the estimate of $\| \D\baru - \D\barv \|_{\LL^p}$.

  Combining \eqref{eq:interp:stab:5}, \eqref{eq:interp:stab:7}, and
  \eqref{eq:interp:stab:20} yields the lower bound in
  \eqref{eq:interp_stab}, for $p \in (1, \infty)$. For $p = 1$ a minor
  variation of the argument gives the same result.
\end{proof}

Finally, we state some useful embedding results.

\begin{lemma}[Embeddings]
  \label{th:embeddings}
  (i) Let $p, q \in [1, \infty]$ and $i \leq j \in \{0, \dots, 3\}$; then
  there exists a constant $C$ such that
  \begin{equation}
    \label{eq:embedding_local}
    \| \D^j \tilu \|_{L^p(Q)} \leq C \| \D^i \tilu \|_{L^q(Q)} \qquad
    \forall u \in \Us, \quad \forall Q \in QQ.
  \end{equation}
  (ii) Moreover, if $p \geq q$, then 
  \begin{equation}
    \label{eq:embedding_global}
    \| \D^j \tilu \|_{L^p(\R^d)} \leq C \| \D^i \tilu \|_{L^q(\R^d)} \qquad
    \forall u \in \Us.
  \end{equation}
\end{lemma}
\begin{proof}
  (i) The first estimate results from local norm-equivalence in each
  element $Q \in \QQ$, due to the fact that the spaces $\{ \tilu|_Q
  \sep u \in \Us \}$ are finite-dimensional and $Q$-independent.

  (ii) The second estimate is a consequence of the embedding $\ell^q
  \subset \ell^p$.
\end{proof}

\subsection{Discrete Sobolev spaces}
\label{sec:sobolev}
We define the discrete Sobolev spaces $\Ws^{1,p}$, $p \in [1,\infty]$,
by
\begin{displaymath}
  \Ws^{1,p} := \b\{ u \in \Us \bsep  \baru \in W^{1,p} \b\},
\end{displaymath}
equipped with the norm $\| u \|_{\Ws^{1,p}} := \| \baru
\|_{W^{1,p}}$. In view of the embedding \eqref{eq:embedding_global}
the discrete Sobolev norm is equivalent to the $\ell^p$-norm, and
hence, $\Ws^{1,p} = \ell^p$. We will introduce the more
natural {\em homogeneous Sobolev spaces} in the next sub-section.

\subsection{Discrete homogeneous Sobolev spaces}
We introduce function spaces that naturally arise in the analysis of
atomistic models and their approximations. We define the (semi-)norm
\begin{displaymath}
  \|u\|_{\Us^{1,p}} := \| \D \baru \|_{\LL^p}, \qquad \text{for } u \in
  \Us, \quad p \in [1, \infty].
\end{displaymath}
Since $\| \cdot \|_{\Us^{1,p}}$ does not penalize translations, we
define the equivalence classes
\begin{displaymath}
  [u] := \{ u + t \sep t \in \R^d\}, \qquad \text{for } u \in \Us.
\end{displaymath}
We will not make the distinction between $u$ and $[u]$, whenever it is
possible to do so without confusion, for example, when all quantities
involved in a statement are translation invariant.

For $p \in [1, \infty]$ we define the discrete function space
\begin{equation}
  \label{eq:defn_Us12}
  \Us^{1,p} := \b\{ [u] \bsep u \in \Us, \| \D \bar u \|_{\LL^p} <
  +\infty \b\},
\end{equation}
which we analyze in the following results. We will also require the
space of displacements for which the displacement gradient has compact
support,
\begin{displaymath}
  \Usz := \b\{ [v] \in \Us \sep {\rm supp}(\D \barv) \text{ is compact}\b\}.
\end{displaymath}

\begin{proposition}
  \label{th:interp:energy-space}
  $\Us^{1,p}$ is a Banach space. For $p \in [1, \infty)$, $\Usz$ is
  dense in $\Us^{1,p}$.
\end{proposition}
\begin{proof}
  The proof for continuous homogeneous Sobolev spaces \cite[Prop. 2.1
  and Thm. 2.1]{OrSu:homsob:2012} translates verbatim. 

  In these results, density of (equivalence classes of) displacements
  with compact support was shown, which is a small class than
  $\Usz$. Hence the case $d = p = 1$ was excluded. Here, we admit the
  slightly larger class and it is straightforward to show that $p = d
  = 1$ now also included in the result. (For $d > 1$ the two classes
  coincide.)
\end{proof}

Next, we analyze the convolution operator $\Conv$ on $\Us^{1,p}$.

\begin{lemma}
  \label{th:inverse_conv}
  Let $p \in [1, \infty)$ then $\Conv : \Us^{1,p} \to \Us^{1,p}$,
  $\Conv u := (\tilu(\xi))_{\xi \in \L}$, is an isomorphism. Moreover,
  we have the stronger estimate
  \begin{equation}
    \label{eq:error_inverse_conv}
    \b\|\Conv^{-1} u - u \b\|_{\ell^p} \leq C \| \D\baru \|_{\LL^p}.
  \end{equation}
\end{lemma}
\begin{proof}
  Fix $u \in \Us^{1,p}$. We need to show that there exists $u^* \in
  \Us^{1,p}$ such that $\tilu^* = u$. This is equivalent to the
  statement that there exists $w \in \Us^{1,p}$ such that $\tilw =
  \tilu - u$. By repeating the argument following
  \eqref{eq:interp:stab:15} it follows that, in fact, $\tilu - u \in
  \ell^p$. Since we know that $\Conv : \ell^p \to \ell^p$ is an
  isomorphism, we can define $w := \Conv^{-1}(\tilu - u)$. Applying
  Lemma \ref{th:Conv_ellp} we obtain
  \eqref{eq:error_inverse_conv}. Since $\|\D \barw \|_{L^p} \leq C \|
  w \|_{\ell^p}$, it follows that $\Conv : \Us^{1,p} \to \Us^{1,p}$ is
  indeed an isomorphism.
\end{proof}

\subsection{A smooth nodal interpolant}
\label{sec:nodal_interp}

\def\Itil{\tilde{I}} In view of Lemma \ref{th:Conv_ellp} and Lemma
\ref{th:inverse_conv} we can now define the nodal interpolant
\begin{equation}
  \label{eq:defn_nodal_interp}
  \Itil u := \widetilde{\Conv^{-1} u}, \qquad \text{for } u \in
  \Us^{1,p} \cup \ell^p, \quad p \in [1, \infty).
\end{equation}
Lemma \ref{th:inverse_conv} immediately
implies the following norm-equivalence result.

\begin{lemma}
  \label{th:norm_equ_Itil}
  There exist constants $\tilde{c}_0, \tilde{c}_1$, such that
  \begin{equation}
    \label{eq:bounds_Du_nodal_interp}
    \tilde{c}_0 \| \D\baru \|_{\LL^p} \leq \|\D\Itil u \|_{\LL^p} 
    \leq \tilde{c}_1 \|\D\baru \|_{\LL^p} \qquad \forall u \in
    \Us^{1,p}, \quad p \in [1, \infty).
  \end{equation}
\end{lemma}

\section{Discrete Deformation and Displacement Spaces}
\label{sec:a:defm_space}
The goal of this short section is to make rigorous the meaning of the
far-field boundary condition for a discrete deformation field $y : \Z^d \to
\R^d$,
\begin{equation}
  \label{eq:bc_vague}
  y(\xi) \sim \mA \xi \quad \text{as } |\xi| \to \infty.
\end{equation}

In order to rigorously define the far-field boundary condition
\eqref{eq:bc_vague} it is useful to analyze the asymptotic bahaviour
of displacements $u \in \Us^{1,p}$.

\begin{proposition}
  Let $p \in [1, \infty)$ and $u \in \Us^{1,p}$. 
  \begin{itemize}
  \item[(i)] If $p > d$, then $|u(\xi) - u(0)| \leq C \| u
    \|_{\Us^{1,p}} |\xi|^{1/p'}$;
  \item[(ii)] If $p = d$, then $|u(\xi) - u(0)| \leq C \| u
    \|_{\Us^{1,p}} (1+\log|\xi|)$;
  \item[(iii)] If $p < d$, then there exists a $u_0 \in [u]$ such that
    $u_0 \in \ell^{p*}$.
  \end{itemize}
\end{proposition}
\begin{proof}
  First, we note that $\Us^{1,p} \subset \Us^{1,\infty}$ for all $p
  \in [1, \infty)$. The result therefore follows directly from
  \cite[Thm. 2.2]{OrSu:homsob:2012}.
\end{proof}

Proposition \ref{th:interp:energy-space} allows us to give a clear
interpretation to the far-field boundary condition
\eqref{eq:bc_vague}. Let $\mA \in \R^{d \times d}_+$, $\yA(x) := \mA
x$ for $x \in \R^d$, and let
\begin{equation}
  \label{eq:defn_YsA}
  \Ys^{1,p} := \b\{ [y] \bsep y \in \Us, [y - \yA] \in \Us^{1,p} \b\}.
\end{equation}
From Proposition \ref{th:interp:energy-space} we infer that, for all
$y \in [y] \in \Ys^{1,p}$,
\begin{displaymath}
  \frac{\b| y(\xi) - \mA \xi \b|}{|\xi|} \to 0, \quad \text{or equivalently,}
  \quad
  y(\xi) = \mA\xi + o(|\xi|),  \quad \text{as } |\xi| \to \infty.
\end{displaymath}
Motivated by this discussion we may confidently take $\Ys^{1,p}$ as
suitable classes of admissible deformations, or equivalently,
$\Us^{1,p}$ as suitable classes of admissible displacements.

\section{Approximation Error Estimates}
\label{sec:interp:errest}
\def\Ibar{\bar{I}}
\def\Jtil{\tilde{J}}
\subsection{Nodal interpolants}
%
In this section we prove two useful interpolation error estimates.  We
define the nodal first-order interpolant and extend the nodal
third-order interpolant, for $v \in C(\R^d; \R^m)$, by
\begin{align*}
  \Ibar v(x) := \sum_{\xi \in \Z^d} v(\xi) \zz(x - \xi), \qquad
  \text{and} \qquad
  \Itil v(x) := (\Itil (v|_{\Z^d}))(x),
\end{align*}
where the latter is well-defined provided that $v|_{\Z^d} \in
\Us^{1,p}$, $p \in [1, \infty]$. Moreover, for $v = \yA + u$, $u \in
\Us^{1,p}$, we can define $\Itil v = \yA + \Itil u$.

\begin{lemma}
  \label{th:interp:errest}
  Let $k \in \{1,2\}$, and $p \in (d/k, \infty]$ if $d > 1$; then
  there exists a constant $C_{\Ibar}$ such that, for all $v \in
  \WW^{k, p}(Q)$, and $Q \in \QQ$,
  \begin{displaymath} 
    \| \Ibar v - v \|_{\WW^{k, p}(Q)} \leq C_{\Ibar} \| \D^k v \|_{\LL^p(Q)}.
  \end{displaymath}
  In particular, if $v \in \WW^{k,p}_\loc$ with $\D^k v \in L^p$, then
  \begin{displaymath}
    \| \Ibar v - v \|_{\WW^{k, p}} \leq C_{\Ibar} \| \D^k v \|_{\LL^p}.
  \end{displaymath}

  (ii) Let $k \in \{1,2,3,4\}$, and $p \in (d/k, \infty]$ if $d > 1$;
  then there exists a constant $C_{\Itil}$ such that, for all $v \in
  \WW^{k, p}_\loc$ with $\D v \in W^{k-1,p}$, $[v|_{\Z^d}] \in
  \Us^{1,p}$ and
  \begin{displaymath}
    \| \Itil v - v \|_{\WW^{k, p}} \leq C_{\Itil} \| \D^k v \|_{\LL^p}.
  \end{displaymath}
\end{lemma}
\begin{proof}
  (i) The first statement follows from a standard Bramble--Hilbert
  argument; the condition $p > d/k$ ($p \geq 1$ if $d = 1$) ensures
  that the nodal values are well-defined and hence the interpolation
  operator is stable; see \cite{Ciarlet:1978} for various examples of
  the argument.

  (ii.1) We first prove that $[v|_{\Z^d}] \in \Us^{1,p}$. From (i) we
  know that $\bar{I} v - v \in W^{k,p}$ and in particular
  $\nabla\bar{I} v \in L^p$. By definition, this implies that
  $[v|_{\Z^d}] \in \Us^{1,p}$. Hence, $\tilde{I} v$ is well-defined.

  (ii.2) The interpolation error estimate for $\tilde{I}$ requires
  some care since the interpolation operator $\Itil$ and hence $\Itil$
  are defined through a linear system, i.e., they are non-local. To
  prove this result we choose an arbirary $w \in \Us^{1,p}$, and
  estimate
  \begin{displaymath}
    \| \Itil v - v \|_{W^{k,p}} \leq \| \Itil v - \tilw \|_{W^{k,p}}
    + \| \tilw - v \|_{W^{k,p}}.
  \end{displaymath}
  Clearly, $\Itil \tilw = \tilw$. Using \eqref{eq:embedding_global} we
  obtain
  \begin{align*}
    \| \Itil v - \tilw \|_{W^{k,p}} =\,& \| \Itil (v - \tilw)
    \|_{W^{k,p}} \leq C \| \Itil (v - \tilw) \|_{L^p} \leq C \|
    \Ibar(v-\tilw) \|_{L^p},
  \end{align*}
  where the last inequality follows from Lemma \ref{th:Conv_ellp} and
  from Lemma \ref{th:norm_equiv_Lp}. The assumption that $p > d/k$ if
  $d > 1$ ensures that $\Ibar$ is stable, that is, $\|
  \Ibar(v-\tilw) \|_{L^p} \leq C \| v - \tilw \|_{W^{k,p}}$. Thus, we
  have obtained that 
  \begin{displaymath}
     \| \Itil v - v \|_{W^{k,p}} \leq C \| \tilw - v \|_{W^{k,p}}.
  \end{displaymath}
  Now choosing $\tilw$ to be a suitable quasi-interpolant gives the
  desired result; this is provided by Lemma \ref{th:Jtil_errest} in
  the next sub-section.
\end{proof}

\begin{remark}
  Let $p > d/k$ if $d > 1$ and arbitrary otherwise (e.g., $p = 2$ and
  $d \in \{1,2,3\}$). Let $u \in \Us^{1,p}$ be a discrete
  displacement, e.g., the equilibrium displacement of an atomistic
  model. Then Lemma \ref{th:interp:errest} (ii) implies that, for any
  function $\hat{u} \in \WW^{k,p}$ interpolating $u$ in lattice sites, we
  have
  \begin{displaymath}
    \| \D^k \Itil u \|_{L^p} \leq C \| \D^k \hat{u} \|_{L^p},
  \end{displaymath}
  for some generic constant $C$. Thus, up to a generic constant,
  $\Itil u$ is the ``smoothest'' interpolant of $u$. In particular,
  $\D^k \Itil u$ is a canonical measure for the smoothness of the
  discrete map $u$.
\end{remark}

\subsection{A quasi-interpolant}
We are left to define and analyze the quasi-interpolant used in the
proof of Lemma \ref{th:interp:errest} (ii). The resulting interpolant
will also be interesting in its own right since it will remove the
restriction $p > d / k$ imposed on the nodal interpolant $\Itil$.

The main ideas of our construction are standard, however, the details
require some care.  The idea, following Cl\'{e}ment \cite{Clement}, is
to construct a bi-orthogonal basis function $\zzz^*$ with compact
support such that
\begin{equation}
  \label{eq:bi-orth_relation}
  \int \zzz^*(x) \zzz(\xi-x) \dx = \cases{1, & \xi = 0, \\ 0, & \text{otherwise,}}
\end{equation}
and to define
\begin{equation}
  \label{eq:defn_quasi_interp}
  \Jtil v(x) := \sum_{\xi \in \Z^d} (\zzz^* \ast v)(\xi) \zzz(x - \xi).
\end{equation}

\begin{lemma}
  \label{th:biorth_bas}
  There exists $\zzz^* \in L^\infty(\R^d)$ with support ${\rm
    supp}\,\zzz^* = \tilde\omega_0$, satisfying
  \eqref{eq:bi-orth_relation}.
\end{lemma}
\begin{proof}
  Let $X := \{ \xi \in \Z^d \sep |\tilde\omega_\xi \cap
  \tilde\omega_0| > 0 \}$, and let $\mathscr{X} := {\rm span}\{
  \zzz_\xi \sep \xi \in X \}$, where $\zzz_\xi = \zzz(\cdot - \xi)$.
  We seek $\zzz^*$ of the form
  \begin{displaymath}
    \zzz^*(x) := \cases{\sum_{\xi \in X} a_\xi \zzz(x - \xi), & x \in
      \tilde\omega_0, \\
      0, & \text{otherwise},}
  \end{displaymath}
  for some parameters $a_\xi$. 

  Since the functions $\zzz_\xi$ are linearly independent in $\R^d$,
  $\zzz_0$ cannot be written as a linear combination of $\zzz_\xi, \xi
  \in X$. Therefore, the mapping 
  \begin{displaymath}
    \mathscr{X} \to \R, \qquad \sum_{\xi \in X} b_\xi \zzz(x - \xi)|_{x \in
      \tilde\omega_0} \mapsto b_0,
  \end{displaymath}
  is well-defined and a linear functional on $\mathscr{X}$. Thus, by
  the Riesz representation theorem, for $\mathscr{X}$ equipped with
  the $L^2$ inner product, there exist coefficients $(a_\xi)_{\xi \in
    X}$ such that
  \begin{displaymath}
    \int_{\tilde\omega_0} \zzz^*(x) \zzz(x-\xi) \dx = \cases{1, & \xi =
      0, \\
      0, & \text{otherwise}.}
  \end{displaymath}
  Since ${\rm supp}(\zzz^*) = \tilde\omega_0$ the result follows.
\end{proof}

\begin{lemma}
  \label{th:tilw_difference}
  Let $\zz_1$ and $\zz_2$ be two different functions satisfying
  (Z1)--(Z3), let $\zzz_i := \zz_i \ast \zz_i$, and let $p(x)$ be a
  cubic polynomial.  Then
  \[
  \sum_{\xi\in\Z^d} p(\xi) \big(\zzz_2(x-\xi)-\zzz_1(x-\xi)\big).
  \]
\end{lemma}
\begin{proof}
Denote $\delta\zz := \zz_2-\zz_1$.
From (Z3) we have that
\[
\sum_{\xi\in\Z^d} \delta\zz(x-\xi) = 0.
\]
Thus, there exist Lipschitz functions $f_k : \R^d \to \R^d$,
$k=1,\ldots,d$, with compact support,
such that
\[
\delta\zz(x) = \sum_{k=1}^d D_k f^{(k)}(x),
\]
where $D_k u(x) := u(x+e_k) - u(x)$.

Further applying summation by parts, we have that, for all $\ell=1,\ldots,d$ and $x\in\R^d$,
\begin{align*}
0
=~&
\sum_{\xi\in\Z^d} \xi_\ell \delta\zz(x-\xi)
 =
\sum_{\xi\in\Z^d} \xi_\ell \sum_{k=1}^d D_k f^{(k)}(x-\xi)
\\ =~&
\sum_{\xi\in\Z^d} \sum_{k=1}^d (D_k \xi_\ell) f^{(k)}(x-\xi)
=
\sum_{\xi\in\Z^d} \sum_{k=1}^d \delta_{k,\ell} f^{(k)}(x-\xi)
= \sum_{\xi\in\Z^d} f^{(\ell)}(x-\xi).
\end{align*}
From the last identity, we have that 
\[
f^{(\ell)}(x) = \sum_{k=1}^d D_k f^{(\ell,k)}(x)
\]
and hence
\[
\delta\zz(x) = \sum_{k,\ell=1}^d D_\ell D_k f^{(\ell,k)}(x).
\]

We can now return to proving the statement of the lemma.
Using the representation
\[
\zzz_2
= \zz_1*\zz_1+2 \zz_1*\delta\zz+\delta\zz*\delta\zz
= \zzz_1+ (2 \zz_1+\delta\zz)*\delta\zz
\]
we have
\begin{align*}
~&
\sum_{\xi\in\Z^d} p(\xi) \big(\zzz_2(x-\xi)-\zzz_1(x-\xi)\big)
\\=~&
\sum_{\xi\in\Z^d} p(\xi) \big((2 \zz_1+\delta\zz)*\delta\zz\big)(x-\xi)
\\=~&
\sum_{\xi\in\Z^d} \sum_{\ell, k=1}^d p(\xi) \big((2 \zz_1+\delta\zz)*D_\ell D_k f^{(\ell,k)}\big)(x-\xi)
\\=~&
\sum_{\xi\in\Z^d} \sum_{\ell, k=1}^d p(\xi) D_\ell D_k \big((2 \zz_1+\delta\zz)*f^{(\ell,k)}\big)(x-\xi)
\\=~&
\sum_{\xi\in\Z^d} \sum_{\ell, k=1}^d (D_\ell D_k p(\xi)) \big((2 \zz_1+\delta\zz)*f^{(\ell,k)}\big)(x-\xi).
\end{align*}
Since $D_\ell D_k p(\xi)$ is a linear function for all $\ell$ and $k$ and $\zz_1$ satisfies (Z3), we have that
\[
\sum_{\xi\in\Z^d} \sum_{\ell, k=1}^d (D_\ell D_k p(\xi)) \big(2
\zz_1*f^{(\ell,k)}\big)(x-\xi) = 2 \sum_{\ell,k = 1}^d \int D_\ell D_k p(x-y) f^{(\ell,k)}(y) \dy
\]
is indeed a linear function of $x$.
%
%
It can be shown that the remaining part is zero:
\begin{align*}
~&
\sum_{\xi\in\Z^d} \sum_{\ell, k=1}^d (D_\ell D_k p(\xi)) \big(\delta\zz*f^{(\ell,k)}\big)(x-\xi)
\\=~&
\sum_{\xi\in\Z^d} \sum_{\ell, k=1}^d \sum_{\ell', k'=1}^d
(D_\ell D_k p(\xi)) \big(D_{\ell'} D_{k'} f^{(\ell',k')}*f^{(\ell,k)}\big)(x-\xi)
\\=~&
\sum_{\xi\in\Z^d} \sum_{\ell, k=1}^d \sum_{\ell', k'=1}^d
(D_{\ell'} D_{k'} D_\ell D_k p(\xi)) \big(f^{(\ell',k')}*f^{(\ell,k)}\big)(x-\xi)
= 0 
\end{align*}
since $D_{\ell'} D_{k'} D_\ell D_k p(\xi)\equiv 0$.
\end{proof}

\begin{lemma}
  \label{th:Itil_quad}
  Let $p(x)$ be a cubic polynomial, then there exists $w \in \Us$
  such that $\tilw = p$.
\end{lemma}
\begin{proof}
  Assume, without loss of generality, that $m = 1$.

  First, suppose, that $p(x) = a^T x + c$ is linear. Let $v(\xi) :=
  a^T\xi + c$, then by (Z3), $\barv(x) = p$. Hence $\tilv = \zz \ast
  \barv = p$ as well, since the convolution with $\zz$ preserves
  linear functions.
  
  Now notice that, if $\zz$ is the Q1-nodal basis function given by
  \eqref{eq:q1_interpolation}, then $\{\tilw \sep w\in\Us\}$, is the
  set of all multi-cubic B-splines and therefore contains $p$
  \cite{Hollig2003}; i.e., $p=\tilv$ for some $v \in \Us$.  For a
  general $\zz$, Lemma \ref{th:tilw_difference} guarantees that $\tilv
  - p$ is a linear function, where $v \in \Us$ is the function we just
  constructed. Hence letting $w := v - (\tilv - p)|_{\Z^d}$ and using
  again the fact that convolution with $\zz$ leaves linear functions
  invariant, yields $\tilw = \tilv - (\tilv - p) = p$.
\end{proof}

\begin{theorem}
  \label{th:Jtil_errest}
  Let $\Jtil$ be defined by \eqref{eq:defn_quasi_interp} with $\zzz^*$
  from Lemma \ref{th:biorth_bas} and let $0 \leq j \leq k \leq 4$;
  then, for all $v \in W^{k, p}_\loc, \D v \in W^{k-1,p}$, we have
  \begin{displaymath}
    \b\|\D^j (\Jtil v - v) \b\|_{L^p} \leq C_{\Jtil} \| \D^k v \|_{L^p}.
  \end{displaymath}
\end{theorem}
\begin{proof}
  For any function $\tilde{w} = \sum_{\xi \in \Z^d} w(\xi) \zzz(x -
  \xi)$, the definition of $\Jtil$ guarantees that $\Jtil \tilw =
  \tilw$, that is, $\Jtil$ is a projector.

  We begin by applying a triangle inequality. For any $Q \in \QQ$, $0
  \leq j \leq k \leq 4$ and $\tilw, w \in \Us$, we have
  \begin{equation}
    \label{eq:Jtil_prf_10}
    \| \D^j(\Jtil v - v) \|_{L^p(Q)} \leq\, \| \D^j (\Jtil v - \tilw)
    \|_{L^p(Q)} + \| \D^j (\tilw - v) \|_{L^p(Q)}. 
  \end{equation}
  
  Using the inverse estimate \eqref{eq:embedding_local}, and a
  straightforward computation (or, a local norm-equivalence argument)
  we obtain
  \begin{displaymath}
    \| \D^j (\Jtil v - \tilw)
    \|_{L^p(Q)} = \| \D^j \Jtil(v - \tilw) \|_{L^p(Q)} \leq C \|
    \Jtil(v-\tilw)\|_{L^p(Q)} \leq C \| \Jtil(v-\tilw) \|_{\ell^p(\Z^d
    \cap \tilde\omega_Q')},
  \end{displaymath}
  where $\tilde\omega_Q' := \bigcup_{\xi\in\Z^d\cap Q}
  \tilde\omega_\xi$.  For each $\xi \in \Z^d \cap \tilde\omega_Q'$, we
  have
  \begin{equation}
    \label{eq:Jtil_prf_25}
    \b| \Jtil(v-\tilw)(\xi)\b| = \bg|\int_{\tilde\omega_\xi} \zzz^*(x-\xi)
    (v - \tilw)(x) \dx\bg| \leq C \| v - \tilw \|_{L^p(\tilde\omega_\xi)}.
  \end{equation}

  Combining \eqref{eq:Jtil_prf_25} with \eqref{eq:Jtil_prf_10}, we
  obtain 
  \begin{displaymath}
    \| \D^j(\Jtil v - v) \|_{L^p(Q)} \leq C \b( \| v - \tilw
    \|_{L^p(\tilde\omega_Q)} + \| \D^j (v - \tilw) \|_{L^p(Q)} \b),
  \end{displaymath}
  where $\tilde\omega_Q \supset \tilde\omega_Q'$ is a compact set of
  uniformly bounded diameter. According to Lemma \ref{th:Itil_quad} we
  may choose $\tilw$ to be any cubic polynomial. Minimizing the
  right-hand side over all cubics yields
 \begin{displaymath}
    \| \D^j(\Jtil v - v) \|_{L^p(Q)} \leq C \| \D^k v \|_{L^p(\tilde\omega_Q)}.
 \end{displaymath}
 Summing over all $Q \in \QQ$ and estimating the overlap of the sets
 $\tilde\omega_Q$ completes the proof.
\end{proof}





\bibliographystyle{plain}
\bibliography{qc}

\end{document}